\documentclass[11pt,twoside, leqno]{article}

\usepackage{amssymb}
\usepackage{amsmath}
\usepackage{mathrsfs}
\usepackage{amsthm}
\usepackage{amsfonts}
\usepackage{enumerate}
\usepackage{esint}
\usepackage{latexsym,amssymb}
\usepackage{color}

\allowdisplaybreaks

\newtheorem{theorem}{Theorem}[section]
\newtheorem{lemma}[theorem]{Lemma}
\newtheorem{corollary}[theorem]{Corollary}

\theoremstyle{definition}
\newtheorem{remark}[theorem]{Remark}
\newtheorem{definition}[theorem]{Definition}
\numberwithin{equation}{section}

\numberwithin{equation}{section}

\begin{document}

\arraycolsep=1pt

\title{\bf\Large
Sharp Hardy-Rellich Type Inequalities Associated with Dunkl Operators
\footnotetext {\hspace{-0.35cm}
2010 {\it Mathematics Subject Classification}. Primary: 26D10;
Secondary: 20F55, 42B37.
\endgraf {\it Key words and phrases}.
Hardy inequalities, Hardy-Rellich inequalities, Dunkl operators, Best constant.
}}
\author{Li Tang, Haiting Chen, Shoufeng Shen and Yongyang Jin\,\footnote{Corresponding author}}
\date{}
\maketitle

\vspace{-0.6cm}

\begin{center}
\begin{minipage}{13.5cm}
{\small {\bf Abstract} \quad
In this paper, we obtained the Dunkl analogy of classical $L^{p}$
Hardy inequality for $p>N+2\gamma$ with sharp constant $\left(\frac{p-N-2\gamma}{p}\right)^{p}$, where $2\gamma$ is the degree of
weight function associated with Dunkl operators, and $L^{p}$ Hardy inequalities
with distant function in some G-invariant domains. Moreover we proved two sharp
Hardy-Rellich type inequalities for Dunkl operators.}
\end{minipage}
\end{center}

%
%
%
%
%

\section{Introduction}
The classical Hardy inequality\\
\begin{equation}\label{}
\int_{\mathbb{R}^{N}}|\nabla u|^{p}dx\geq\left|\frac{N-p}{p}\right|^{p}\int_{\mathbb{R}^{N}}\frac{|u|^{p}}{|x|^{p}}dx
\end{equation}
holds for $u\in{C}^{\infty}_{0}(\mathbb{R^{N}})$ when $1<p<N$ and for $u\in{C}^{\infty}_{0}(\mathbb{R^{N}}\backslash\{0\})$ when $N<p<\infty$. It has extensive application in analysis, partial differential equation and physical research. In \cite{a}, Hardy firstly proved this inequality in the case of one dimension. Since then, many researchers devoted themselves to it and made great progress, not only in Euclidean spaces, there are counterparts in Riemannian manifolds and Carnot groups, see \cite{b, d, g, l, m} and the references therein.

If $\mathbb{R}^{N}$ is replaced by a bounded convex domain $\Omega$, the following sharp inequality holds for $1<p<\infty$\\
\begin{equation}\label{}
\int_{\Omega}|\nabla u|^{p}dx\geq\left(\frac{p-1}{p}\right)^{p}\int_{\Omega}\frac{|u|^{p}}{{\delta}^{p}}dx
\end{equation}
where $\delta(x):=dist(x,\partial\Omega)$, see \cite{o}. Maz'ya proved in \cite{e} that (1.2) can be characterised in terms of p-capacity. When $\Omega$ is non-convex, the problem is more complicated. For domains such that $-\Delta\delta$ is nonnegative in the distributional sense, some results were obtained by Barbatis, Filippas and Tertikas in \cite{u}. It is equivalent between non-negativity of $-\Delta\delta$ in the distributional sense and the mean-convexity of the domain when the boundary is smooth enough, see \cite{p, v, w, x}. In \cite{f}, Ancona obtained some results in planar simply connected domains by using Koebe one-quarter theorem; some other Hardy inequalities for special domains see \cite{g}.

The Rellich inequality\\
\begin{equation}
\int_{\mathbb{R}^{N}}|\Delta u|^{2}dx\geq\frac{N^{2}(N-4)^{2}}{16}\int_{\mathbb{R}^{N}}\frac{|u|^{2}}{|x|^{4}}dx,
\end{equation}
is a generalisation of Hardy inequality, which holds for $u\in{C}^{\infty}_{0}(\mathbb{R}^{N})$ and the constant $\frac{N^{2}(N-4)^{2}}{16}$ is sharp when $N\geq5$. In \cite{k}, Tertikas and Zographopoulos obtained a Hardy-Rellich type inequality which read as\\
\begin{equation}\label{}
\int_{\mathbb{R}^{N}}|\Delta u|^{2}dx\geq\frac{N^{2}}{4}\int_{\mathbb{R}^{N}}\frac{|\nabla u|^{2}}{|x|^{2}}dx,
\end{equation}
where $N\geq5$, the constant $\frac{N^{2}}{4}$ is also sharp.

In the setting of Dunkl operators, the author In \cite{h}, proved a sharp analogical inequality of (1.1) for
\begin{equation}
\int_{\mathbb{R}^{N}}|\nabla_{k} u|^{2}d\mu_{k}\geq\left(\frac{N+2\gamma-2}{2}\right)^{2}\int_{\mathbb{R}^{N}}\frac{|u|^{2}}{|x|^{2}}d\mu_{k},
\end{equation}
and the following inequality for $p\neq{N+2\gamma}$ and small $\gamma$
\begin{equation}
\int_{\mathbb{R}^{N}}|\nabla_{k} u|^{p}d\mu_{k}\geq{C}\int_{\mathbb{R}^{N}}\frac{|u|^{p}}{|x|^{p}}d\mu_{k},
\end{equation}
however the best constant for $p\neq2$ in (1.6) is not known. They also obtained an analogical inequality of (1.3) for Dunkl Laplacian
\begin{equation}
\int_{\mathbb{R}^{N}}|\Delta_{k} u|^{2}d\mu_{k}\geq\frac{(N+2\gamma)^{2}(N+2\gamma-4)^{2}}{16}\int_{\mathbb{R}^{N}}\frac{|u|^{2}}{|x|^{4}}d\mu_{k},
\end{equation}
where the constant $\frac{(N+2\gamma)^{2}(N+2\gamma-4)^{2}}{16}$ is sharp.

The plan of this paper is as follows: we introduce some definitions and basic facts of Dunkl operators in the second section. Then, in section three we obtained some $L^{p}$ Hardy inequalities associated with distant function for Dunkl operators by choosing specific vector fields, especially an Hardy inequality on a non-convex domain $\Omega={B}(0,R)^{c}$, which leads to classical sharp $L^{p}$-Hardy inequality associated with Dunkl operators for $p>N+2\gamma$. In the last section we obtained two sharp Hardy-Rellich type inequalities for Dunkl operators by the method of spherical h-harmonic decomposition.

\section{ Preliminaries }
Dunkl theory is a generalisation of Fourier analysis and special function theory about root system. It generalized Bessel functions on flat symmetric spaces, also Macdonald polynomials on affine buildings. Moreover, Dunkl theory has extensive application in algebra (double affine Hecke algebras), probability theory (Feller processes with jump) and mathematical physics (quantum many body problems, Calogero-Moser-Sutherland molds).

In this section we will introduce some fundamental concepts and notations of Dunkl operators, see also \cite{i}, \cite{n} for more details.
\\\indent
If a finite set $R\subset \mathbb{R}^{N}\setminus\{0\}$ such that $R\cap\alpha\mathbb{R}=\{-\alpha,\alpha\}$ and $\sigma_{\alpha}(R)=R$ for all $\alpha\in{R}$, then we call $R$ a root system. Denote $\sigma_{\alpha}$ as the reflection on the hyperplane which is orthogonal to the root $\alpha$, written as\\
\begin{equation*}
\sigma_{\alpha}x=x-2\frac{\langle\alpha,x\rangle}{\langle\alpha,\alpha\rangle}\alpha.
\end{equation*}
We write $G$ as the group generated by all the reflections $\sigma_{\alpha}$ for $\alpha\in{R}$, it is a finite group. Let ${k}:R\longrightarrow[0,\infty)$ be a G-invariant function, i.e.,
$k(\alpha)=k(v\alpha)$ for all $v\in{G}$ and all $\alpha\in{R}$, simply written $k_{\alpha}=k(\alpha)$. $R$ can be denoted as $R=R_{+}\cup(-R_{+})$, when $\alpha\in R_{+}$, then$-\alpha\in-R_{+}$, and $R_{+}$ is called a positive subsystem. We fix a positive subsystem $R_{+}$ in a root system $R$. Without loss of generality we assume that $|\alpha|^{2}=2$ for all $\alpha\in{R}$.
\begin{definition}
For $i=1,...,N,$ the Dunkl operators on $C^{1}(\mathbb{R}^{N})$ is defined as follows
\begin{equation*}
T_{i}u(x)=\partial_{i}u(x)+\sum\limits_{\alpha\in{R_{+}}}k_{\alpha}\alpha_{i}\frac{u(x)-u(\sigma_{\alpha}x)}{\langle\alpha,x\rangle}.
\end{equation*}
\end{definition}
By this definition we can see that even if the decomposition of $R$ is not unique, the different choice of positive subsystem make no difference in the definitions due to the G-invariance of $k$.
Denote by $\nabla_{k}=(T_{1},\ldots,T_{N})$ the Dunkl gradient, $\Delta_{k}=\sum\limits_{i=1}^{N}T^{2}_{i}$ the Dunkl-Laplacian. Especially, for $k=0$ we have $\nabla_{0}=\nabla$ and $\Delta_{0}=\Delta$. The Dunkl-Laplacian can be written in terms of the usual gradient and Laplacian as follows,
\begin{equation*}
\Delta_{k}u(x)=\Delta{u(x)}+2\sum\limits_{\alpha\in{R_{+}}}k_{\alpha}\left[\frac{\langle\nabla{u(x)},\alpha\rangle}{\langle\alpha,x\rangle}-\frac{u(x)-u(\sigma_{\alpha}x)}{\langle\alpha,x\rangle^{2}}\right].
\end{equation*}
The weight function naturally associated to Dunkl operators is
\begin{equation*}
\omega_{k}(x)=\prod_{\alpha\in{R_{+}}}|\langle\alpha,x\rangle|^{2k_{\alpha}}.
\end{equation*}
This is a homogeneous function of degree $2\gamma$, where
\begin{equation*}
\gamma:= \sum\limits_{\alpha\in{R_{+}}}k_{\alpha}.
\end{equation*}
We will work in spaces $L^{p}(\mu_{k})$, where $d\mu_{k}=\omega_{k}(x)dx$ is the weighted measure. About this weighted measure we have the formula of integration by parts
\begin{equation*}
\int_{{\mathbb{R}}^{N}}T_{i}(u)vd\mu_{k}=-\int_{{\mathbb{R}}^{N}}uT_{i}(v)d\mu_{k}.
\end{equation*}
If at least one of the functions $u$, $v$ is G-invariant, the following Leibniz rule
\begin{equation*}
T_{i}(uv)=uT_{i}v+vT_{i}u,
\end{equation*}
holds. In general we have
\begin{equation*}
T_{i}(uv)(x)=v(x)T_{i}u(x)+u(x)T_{i}v(x)-\sum\limits_{\alpha\in{R_{+}}}k_{\alpha}\alpha_{i}
\frac{(u(x)-u(\sigma_{\alpha}x))(v(x)-v(\sigma_{\alpha}x))}{\langle\alpha,x\rangle}.
\end{equation*}

\section{ $L^{p}$ Hardy inequalities }
In this section we proved a general Hardy inequality with remainder terms for Dunkl operators in G invariant domains, then we get the Dunkl analogy of Hardy inequality (1.1) for $p>N+2\gamma$.

Firstly, we review some basic facts of distant function.
\begin{lemma}(\cite{j})
Let $\Omega\subset\mathbb{R}^N$ be an open set such that $\partial\Omega\neq \emptyset$. The following propositions hold true.\\
$\emph{(i)}$  The function $\delta(x)$ is differentiable at a point $x\in\Omega$ if and only if there exists a unique point $N(x)=y\in\partial\Omega$ such that $\delta(x)=|x-y|$. If $\delta(x)$ is differentiable, then $\nabla{\delta}(x)=\frac{x-y}{|x-y|}$, and $|\nabla{\delta}|=1$.\\
$\emph{(ii)}$  Denote $\Sigma(\Omega)$ as the the set of points where $\delta(x)$ is not differentiable. If $\Omega$ is bounded and with $C^{2,1}$ boundary, then $|\Sigma(\Omega)|=0$. \\
$\emph{(iii)}$  Assume that $\Omega$ is convex. Then $\Delta{\delta}\leq {0}$ in the sense of distributions, i.e.,
\end{lemma}
\centerline{$\int_{\Omega}\delta(x)\Delta{\varphi}(x)dx\leq{0}, \varphi\in{C}_{0}^{\infty}(\Omega), \varphi\geq{0}.$}
For $\Omega\subset\mathbb{R}^{N}$, if for all $x\in{\Omega}$, $g\in{G}$, we have $gx\in{\Omega}$, then $\Omega$ is called a G-invariant domain.
\begin{lemma}\label{lm-2}
If  $\Omega\subset\mathbb{R}^{N}$ is G-invariant, $g\in{G}$, $x\in{\Omega}\setminus\Sigma(\Omega)$, then
\begin{equation}
(\nabla{\delta}\circ{g})(x)=(g\circ{\nabla{\delta}})(x).
\end{equation}
\end{lemma}
\begin{proof}
From the proof of Theorem 5.2 in [8], the function $\delta(x)$ is G-invariant. For any $x\in{\Omega}$, we have
$y=N(x)\in\partial{\Omega}$ , $\delta{(x)}=|x-y|$ and
\begin{equation*}
\delta{(gx)}=\delta(x)=|x-y|=|g(x-y)|=|gx-gy|.
\end{equation*}
Due to the uniqueness of $N(x)$, we get that $N(gx)=gy$. Therefore
\begin{equation*}
\nabla{\delta{(gx)}}=\frac{gx-gy}{|gx-gy|}=\frac{g(x-y)}{|x-y|}=g({\nabla{\delta(x)}}).
\end{equation*}
\end{proof}
\begin{remark}
{}If  $F=h_{1}x+h_{2}\nabla{\delta}$, where $h_{1}$, $h_{2}$ are G-invariant functions, then by Lemma \ref{lm-2} we have that
$\langle\alpha,F(\sigma_{\alpha}x)\rangle=-\langle\alpha,F\rangle$.
\end{remark}
\begin{theorem}\label{T-3-1}
Let $\Omega\subset\mathbb{R}^{N}$ be a G-invariant domain with $|\Sigma(\Omega)|=0$. Then for all $u\in C^{\infty}_{0}(\Omega)$ we have the inequality
\begin{equation}
\begin{split}
\int_{\Omega}|\nabla_{k}u|^{2}d\mu_{k}\geq&\left(\frac{p-1}{p}\right)^{p}\int_{\Omega}\frac{{|u|}^{p}}{{\delta}^{p}}d\mu_{k} \\
    &+\left(\frac{p-1}{p}\right)^{p-1}\int_{\Omega}\left[-\Delta{\delta}+(\frac{p}{2}-1)\langle\rho,{\nabla\delta}\rangle-\frac{p}{2}|\langle\rho,{\nabla\delta}\rangle|\right]\frac{{|u|}^{p}}{{\delta}^{p-1}}d\mu_{k},
\end{split}
\end{equation}
where $\rho:=2\sum\limits_{\alpha\in{{R}_{+}}}k_{\alpha}\frac{\alpha}{\langle\alpha,x\rangle}$.
\end{theorem}
\begin{proof}
If $F$ satisfies that $\langle\alpha,F(\sigma_{\alpha}x)\rangle=-\langle\alpha,F\rangle$, then
\begin{equation}\begin{split}\label{T-3-1-1}
  \int_{\Omega}(\nabla_{k}\cdot{F}){|u|}^{p} d\mu_{k} &=-\int_{\Omega}F\cdot\nabla_{k}({|u|}^{p})d\mu_{k}  \\
&=-\int_{\Omega}F\cdot\nabla({|u|}^{p})d\mu_{k}-\int_{\Omega}\sum\limits_{\alpha\in{{R}_{+}}}k_{\alpha}\frac{{|u|}^{p}-{|u(\sigma_{\alpha}x)|}^{p}}{\langle\alpha,x\rangle}\langle\alpha,F\rangle d\mu_{k}.
\end{split}
\end{equation}
Let $x=\sigma_{\alpha}y$,
\begin{equation}
\int_{\Omega}{|u|}^{p}\frac{\langle\alpha,F\rangle}{\langle\alpha,x\rangle}d\mu_{k}=\int_{\Omega}{|u(\sigma_{\alpha}y)|}^{p}\frac{\langle\alpha,F(\sigma_{\alpha}y)\rangle}{\langle\alpha,\sigma_{\alpha}y\rangle}d\mu_{k}(\sigma_{\alpha}y).
\end{equation}
Because of $\langle\alpha,F(\sigma_{\alpha}y)\rangle=-\langle\alpha,F(y)\rangle$, $\langle\alpha,\sigma_{\alpha}y\rangle=-\langle\alpha,y\rangle$,
$d\mu_{k}(\sigma_{\alpha}y)=\omega_{k}(\sigma_{\alpha}y)d(\sigma_{\alpha}y)=\omega_{k}(y)|J|dy$,
where
\begin{equation*}
J=\left|\frac{\partial(\sigma_{\alpha}y)}{\partial(y)}\right|=
\left|
\begin{array}{ccccc}
1-{\alpha^{2}_{1}}&-\alpha_{1}\alpha_{2}&-\alpha_{1}\alpha_{3}&\cdots&-\alpha_{1}\alpha_{n} \\
-\alpha_{2}\alpha_{1}&1-{\alpha^{2}_{2}}&-\alpha_{2}\alpha_{3}&\cdots&-\alpha_{2}\alpha_{n} \\
-\alpha_{3}\alpha_{1}&-\alpha_{3}\alpha_{2}&1-{\alpha^{2}_{3}}&\cdots&-\alpha_{3}\alpha_{n} \\
\vdots&\vdots&\vdots&\ddots&\vdots\\
-\alpha_{n}\alpha_{1}&-\alpha_{n}\alpha_{2}&-\alpha_{n}\alpha_{3}&\cdots&1-{\alpha^{2}_{n}}
\end{array}
\right|.
\end{equation*}
Straightforward calculation shows that $J=-1$. \\
Thus $d\mu_{k}(\sigma_{\alpha}y)=d\mu_{k}(y)$, then
\begin{equation}\label{T-3-1-2}
\int_{\Omega}{|u|}^{p}
{\frac{\langle\alpha,F\rangle}{\langle\alpha,x\rangle}}d\mu_{k}=\int_{\Omega}{|u(\sigma_{\alpha}y)|}^{p}\frac{\langle\alpha,F(y)\rangle}{\langle\alpha,y\rangle}d\mu_{k}(y).
\end{equation}
Putting (\ref{T-3-1-2}) into (\ref{T-3-1-1}), we get
\begin{equation}
\begin{split}
\int_{\Omega}(\nabla_{k}\cdot{F}){|u|}^{p} d\mu_{k}=&-\int_{\Omega}p{|u|}^{p-2}u(F\cdot\nabla{u})d\mu_{k} \\
=&-\int_{\Omega}p{|u|}^{p-2}u\left((F\cdot\nabla_{k}u)-\sum\limits_{\alpha\in{{R}_{+}}}k_{\alpha}\frac{u(x)-u(\sigma_{\alpha}x)}{\langle\alpha,x\rangle}\langle\alpha,F\rangle\right)d\mu_{k} \\
=&-\int_{\Omega}p{|u|}^{p-2}u(F\cdot{\nabla_{k}u}) d\mu_{k}+p\sum\limits_{\alpha\in{{R}_{+}}}k_{\alpha}\int_{\Omega}\frac{\langle\alpha,F\rangle}{\langle\alpha,x\rangle}{|u|^{p}}d\mu_{k} \\
&-p\sum\limits_{\alpha\in{{R}_{+}}}k_{\alpha}\int_{\Omega}\frac{\langle\alpha,F\rangle}{\langle\alpha,x\rangle}{|u|^{p-2}}u\cdot{u(\sigma_{\alpha}x)}d\mu_{k}\\
\leq&{p}\left(\frac{p-1}{p}\epsilon^{-\frac{p}{p-1}}\int_{\Omega}{|F|}^{\frac{p}{p-1}}|u|^{p}d\mu_{k}+\frac{\epsilon^{p}}{p}\int_{\Omega}|\nabla_{k}u|^{p}d\mu_{k}\right) \\
&+\frac{1}{2}p\int_{\Omega}\langle\rho,F\rangle{|u|^{p}} d\mu_{k}-\frac{1}{2}p\int_{\Omega}\langle\rho,F\rangle{|u|^{p-2}}u\cdot{u(\sigma_{\alpha}x)}d\mu_{k},
\end{split}
\end{equation}
we used $H\ddot{o}lder$ inequality and Young inequality in the last inequality above. Then,
\begin{equation}
\begin{split}\label{T-3-1-3}
\int_{\Omega}\epsilon^{p}{|\nabla_{k}u|}^{p}d\mu_{k}\geq&\int_{\Omega}\left(\nabla_{k}\cdot{F}-(p-1){\epsilon^{-\frac{p}{p-1}}}{|F|}^{\frac{p}{p-1}}-\frac{1}{2}p\langle\rho,F\rangle\right){|u|}^{p}d\mu_{k} \\
&+\frac{1}{2}p\int_{\Omega}\langle\rho,F\rangle{|u|^{p-2}}u\cdot{u(\sigma_{\alpha}x)}d\mu_{k}.
\end{split}
\end{equation}

Let $F=-\frac{\nabla{\delta}}{{\delta}^{p-1}}$. Since $\delta$ is G invariant,  $\nabla_{k}{\delta}=\nabla\delta$,
thus $\nabla_{k}\cdot{F}=-\frac{\Delta_{k}\delta}{\delta^{p-1}}+(p-1)\frac{{|\nabla{\delta}|}^{2}}{\delta^{p}}$.
By (\ref{T-3-1-3}), we have

\begin{equation}
\begin{split}\label{T-3-1-4}
\int_{\Omega}{|\nabla_{k}u|}^{p}d\mu_{k}\geq\left(\frac{p-1}{{\epsilon}^{p}}-\frac{p-1}{{\epsilon}^{p+\frac{p}{p-1}}}\right)\int_{\Omega}\frac{|u|^{p}}{\delta^{p}}d\mu_{k} &+\frac{1}{\epsilon^{p}}\int_{\Omega}\left(-\Delta_{k}\delta+\frac{p}{2}\langle\rho,\nabla{\delta}\rangle\right)\frac{|u|^{p}}{\delta^{p-1}}d\mu_{k} \\
&-\frac{p}{2\epsilon^{p}}\int_{\Omega}\langle\rho,\nabla\delta\rangle\frac{{|u|^{p-2}}u\cdot{u(\sigma_{\alpha}x)}}{\delta^{p-1}}d\mu_{k} \\
\geq\left(\frac{p-1}{{\epsilon}^{p}}-\frac{p-1}{{\epsilon}^{p+\frac{p}{p-1}}}\right)\int_{\Omega}\frac{|u|^{p}}{\delta^{p}}d\mu_{k}+\frac{1}{\epsilon^{p}}\int_{\Omega}&\left(-\Delta_{k}\delta+\frac{p}{2}\langle\rho,\nabla{\delta}\rangle-\frac{p}{2}|\langle\rho,\nabla{\delta}\rangle|\right)\frac{|u|^{p}}{\delta^{p-1}}d\mu_{k}.   \end{split}
\end{equation}
The last inequality above is obtained by using $H\ddot{o}lder$ inequality
\begin{equation*}
\begin{split}
\int_{\Omega}\frac{\langle\rho,\nabla\delta\rangle}{\delta^{p-1}}{|u|^{p-2}}u\cdot{u(\sigma_{\alpha}x)}d\mu_{k}
&
\leq\left(\int_{\Omega}\frac{|\langle\rho,\nabla\delta\rangle|}{\delta^{p-1}}{|u|^{p}}d\mu_{k}\right)^{\frac{p-1}{p}}\left(\int_{\Omega}\frac{|\langle\rho,\nabla\delta\rangle|}{\delta^{p-1}}{|u(\sigma_{\alpha}x)|^{p}}d\mu_{k}\right)^{\frac{1}{p}}
\\
&=\int_{\Omega}\frac{|\langle\rho,\nabla\delta\rangle|}{\delta^{p-1}}{|u|^{p}}d\mu_{k}.
\end{split}
\end{equation*}
The $\frac{p-1}{\epsilon^{p}}-\frac{p-1}{\epsilon^{p+\frac{p}{p-1}}}$ takes the maximum value $(\frac{p-1}{p})^{p}$ when
$\epsilon=({\frac{p}{p-1}})^{\frac{p-1}{p}}$. Also,
\begin{equation*}
-\Delta_{k}\delta+\frac{p}{2}\langle\rho,\nabla\delta\rangle=-\Delta\delta+(\frac{p}{2}-1)\langle\rho,\nabla\delta\rangle,
\end{equation*}
we thus completed the proof of Theorem \ref{T-3-1}.
\end{proof}

\begin{remark}
If the root system $\tilde{R}$ satisfies $span(\tilde{R})\subset{{\mathbb{R}}^{N-1}}$. Then the following inequality holds for any $u\in C^{\infty}_{0}({\mathbb{R}}^{N-1}\times{{\mathbb{R}}_{+}})$,

\begin{equation*}
\int_{{\mathbb{R}}^{N-1}\times{{\mathbb{R}}_{+}}}|\nabla_{k}u|^{p}d\mu_{k}\geq
\left(\frac{p-1}{p}\right)^{p}\int_{{\mathbb{R}}^{N-1}\times{{\mathbb{R}}_{+}}}\frac{|u|^{p}}{{x^{p}_{n}}}d\mu_{k}.
\end{equation*}
\end{remark}

Let $S_{N}$ denote the symmetric group in N elements. A root system of $S_{N}$ is given by $R=\{\pm(e_{i}-e_{j}),1\leq i<j\leq N\}$, and
\begin{equation*}
(span(R))^{\perp}=e_{1}+\cdots+e_{N}=:\eta,
\end{equation*}
see \cite{j} for more details. Let the domain $\Omega=span(R)\times{\eta_{+}}$, where $\eta_{+}$ is the positive direction of the straight line coincide with $\eta$. Then $\Omega$ is G invariant, $\delta(x)=dist(x,span{R})$, and
\begin{equation*}
\nabla\delta=\frac{e_{1}+\cdots+e_{N}}{|e_{1}+\cdots+e_{N}|}=\frac{\eta}{\sqrt{N}}.
\end{equation*}
Fix $R_{+}=\{e_{i}-e_{j},1\leq i<j\leq N\}$, then $-\Delta\delta=0$ and $\langle\rho,\nabla\delta\rangle=0$,
by Theorem \ref{T-3-1}, we have:

\begin{corollary}
For $R=\{\pm(e_{i}-e_{j}),1\leq i<j\leq N\}$, $u\in C^{\infty}_{0}({span(R)}\times{\eta_{+}})$, the following inequality holds
\begin{equation*}
\int_{{span(R)}\times{\eta_{+}}}\left|\nabla u+k\sum\limits_{1\leq i<j\leq N}\frac{u(x)-u(\widetilde{x}_{ij})}{x_{i}-x_{j}}(e_{i}-e_{j})\right|^{p}d\mu_{k}\geq
\left(\frac{p-1}{p}\right)^{p}\int_{span(R)\times{\eta_{+}}}\frac{|u|^{p}}{{\delta^{p}}}d\mu_{k},
\end{equation*}
where $k=k_{\alpha}=k_{\beta}, \forall\alpha,\beta\in{R}$, $\widetilde{x}_{ij}=(x_{1},\cdots,x_{i-1},x_{j},x_{i+1},\cdots,x_{j-1},x_{i},x_{j+1},\cdots,x_{N})$.
\end{corollary}
\begin{proof}
 It is easy to prove  $v\circ\sigma_{\alpha}\circ v^{-1}=\sigma_{v\alpha}$, for all $v\in G$, As there is one conjugate class in $R$, so $k_{\alpha}=k_{\beta}$, for all $\alpha,\beta \in R$, see also \cite{i}. Straightforward computation shows $\sigma_{e_{i}-e_{j}}(x)=\widetilde{x}_{ij}$.
\end{proof}

By inequality (\ref{T-3-1-4}) in the proof of Theorem \ref{T-3-1}, it is easy to see that the following theorem holds.
\begin{theorem}\label{T-3-2}
If $\Omega\subset\mathbb{R}^{N}$ satisfies $|\Sigma(\Omega)|=0$, $\langle\rho,\nabla\delta\rangle\geq0$. The following inequality holds for all $u\in C^{\infty}_{0}(\Omega)$,
\begin{equation}
\int_{\Omega}|\nabla_{k}u|^{p}d\mu_{k}\geq(p-1)\left(\epsilon^{-p}-\epsilon^{-\frac{p^{2}}{p-1}}\right)\int_{\Omega}\frac{|u|^{p}}{\delta^{p}}d\mu_{k}-\epsilon^{-p}\int_{\Omega}\Delta_{k}\delta\frac{|u|^{p}}{\delta^{p-1}}d\mu_{k},
\end{equation}
where $\epsilon$ is a positive constant.
\end{theorem}

\begin{remark}
If a domain $\Omega$ satisfies that $|\Sigma(\Omega)|=0$, $\langle\rho,\nabla\delta\rangle\geq0$ and $\delta\Delta_{k}\delta\leq\theta<p-1$, where $\theta$ is a positive constant, i.e. then there is a positive constant $C=C(\theta, p)$ such that
\begin{equation*}
\int_{\Omega}|\nabla_{k}u|^{p}d\mu_{k}\geq C\int_{\Omega}\frac{|u|^{p}}{\delta^{p}}d\mu_{k}.
\end{equation*}
\end{remark}

\begin{corollary}\label{COR-1}
Suppose that $\Omega=B(o,r)^{c}$, $p>N+2\gamma$. The following inequality holds for all $u\in C^{\infty}_{0}(\Omega)$,
\begin{equation}
\int_{\Omega}|\nabla_{k}u|^{p}d\mu_{k}\geq\left(\frac{p-N-2\gamma}{p}\right)^{p}\int_{\Omega}\frac{|u|^{p}}{\delta^{p}}d\mu_{k}.
\end{equation}

\end{corollary}
\begin{proof}
When $\Omega=B(o,r)^{c}$, then $|\Sigma(\Omega)|=0$,  $\delta=|x|-r$, $\nabla\delta=\frac{x}{|x|}$, $\Delta_{k}\delta=\frac{N+2\gamma-1}{|x|}<\frac{p-1}{\delta}$, and
\begin{equation*}
\begin{split}
 (p-1)\left(\epsilon^{-p}-\epsilon^{-\frac{p^{2}}{p-1}}\right)\frac{1}{\delta}-\epsilon^{-p}\Delta_{k}\delta&=(p-1)\left(\epsilon^{-p}-\epsilon^{-\frac{p^{2}}{p-1}}\right)\frac{1}{|x|-r}-\epsilon^{-p}\frac{N+2\gamma-1}{|x|} \\
    &\geq \left[(p-N-2\gamma)\epsilon^{-p}-(p-1)\epsilon^{-\frac{p^{2}}{p-1}}\right]\frac{1}{|x|-r},
\end{split}
\end{equation*}
we note $(p-N-2\gamma)\epsilon^{-p}-(p-1)\epsilon^{-\frac{p^{2}}{p-1}}$ takes the maximum value $\left(\frac{p-N-2\gamma}{p}\right)^{p}$ when $\epsilon=\left(\frac{p}{p-N-2\gamma}\right)^{\frac{p-1}{p}}$.
Note that $\langle\rho,\nabla\delta\rangle=\frac{2\gamma}{|x|}\geq0$, we complete the proof  by  Theorem \ref{T-3-2}.
\end{proof}

Let $r$ tends to zero, the following sharp inequality follows from Corollary \ref{COR-1}.
\begin{corollary}
Suppose that $p>N+2\gamma$. The following inequality holds for all $u\in C^{\infty}_{0}(\mathbb{R}^{N}\backslash\{0\})$,
\begin{equation}
\int_{\mathbb{R}^{N}}|\nabla_{k}u|^{p}d\mu_{k}\geq\left(\frac{p-N-2\gamma}{p}\right)^{p}\int_{\mathbb{R}^{N}}\frac{|u|^{p}}{|x|^{p}}d\mu_{k}.
\end{equation}
\end{corollary}
\begin{proof}
There only remains to prove the optimality of the constant $\left(\frac{p-N-2\gamma}{p}\right)^{p}$. For any $\epsilon>0$ we choose
\begin{equation*}
u_{\epsilon}(r)=\left\{
\begin{array}{cc}
 1 &,r\leq1,\\
 r^{\frac{p-N-2\gamma-\epsilon}{p}} &,r>1.
\end{array}\right.
\end{equation*}
We can write $d\mu_{k}=r^{2\gamma}w_{k}(\xi)drd\nu(\xi)$, where $\nu$ is the surface measure on the sphere ${\mathbb{S}}^{N-1}$. Thus by directly computing we have
\begin{equation*}
\lim\limits_{\epsilon\rightarrow0}\frac{\int_{{\mathbb{R}}^{N}}|\nabla_{k}u|^{p}d\mu_{k}}{\int_{{\mathbb{R}}^{N}}\frac{|u|^{p}}{|x|^{p}}d\mu_{k}}=\left(\frac{p-N-2\gamma}{p}\right)^{p}.
\end{equation*}

\end{proof}

\section{Hardy-Rellich type inequality}

\noindent\textbf{Spherical h-harmonics.}
We will introduce some concepts and fundermental facts for spherical $h$-harmonic theory, see \cite{i} for more details. If a homogeneous polynomial $p$ of degree $n$ that satisfies
\begin{equation*}
\Delta_{k}p=0,
\end{equation*}
then we called it an h-harmonic polynomial of degree $n$. Spherical $h$-harmonics (or just $h$-harmonics) of degree $n$ are defined as the restrictions of $h$-harmonic polynomials of degree $n$ to the unit sphere ${\mathbb{S}}^{N-1}$. Denote $\mathcal{P}_{n}$ the space of $h$-harmonics of degree $n$. Denote $d(n)$ the dimension of $\mathcal{P}_{n}$, it is finite and given by following formula:
\begin{equation*}
d(n)=\tbinom{n+N-1}{N-1}-\tbinom{n+N-3}{N-1}.
\end{equation*}

Moreover, the space ${{L}}^{2}({\mathbb{S}}^{N-1},\omega_{k}(\xi)d\xi)$ can be decomposed as the orthogonal direct sum of the spaces $\mathcal{P}_{n}$, for $n=0,1,2,\ldots$.

Let $Y_{i}^{n},\ i=1,\ldots,d(n)$ be an orthogonal basis of $\mathcal{P}_{n}$, In spherical polar coordinates $x=r\xi$, for $r\in{[0,\infty)}$ and $\xi\in{\mathbb{S}}^{N-1}$, we can write the Dunkl laplacian as
\begin{equation*}
\Delta_{k}=\frac{\partial^{2}}{\partial{r}^{2}}+\frac{N+2\gamma-1}{r}\frac{\partial}{\partial{r}}+\frac{1}{r^{2}}\Delta_{k,0},
\end{equation*}
where $\Delta_{k,0}$ is an analogue of the classical Laplace-Beltrami operator on the sphere, and it only acts on the $\xi$ variable. Then the spherical h-harmonics $Y_{i}^{n}$ are eigenfunctions of $\Delta_{k,0}$, and it's eigenvalues are given by
\begin{equation*}
\Delta_{k,0}Y_{i}^{n}=-n(n+N+2\gamma-2)Y_{i}^{n}=:\lambda_{n}Y_{i}^{n}.
\end{equation*}
The $h$-harmonic expansion of a function $u\in{L^{2}(\mu_{k})}$ can be expressed as
\begin{equation*}
u(r\xi)=\sum_{n=0}^{\infty}\sum_{i=1}^{d(n)}u_{n,i}(r)Y_{i}^{n}(\xi),
\end{equation*}
where
\begin{equation}\label{eq4.1}
u_{n,i}(r)=\int_{{\mathbb{S}}^{N-1}}u(r\xi)Y_{i}^{n}(\xi)\omega_{k}(\xi)d\nu(\xi),
\end{equation}
and $\nu$ is the surface measure on the sphere ${\mathbb{S}}^{N-1}$.

\begin{theorem}\label{T-4-1}
Let $\overline{N}\neq{2}$. Then we have the inequality
\begin{equation}\label{eq4.2}
\int_{{\mathbb{R}}^{N}}|x|^{2}|\Delta_{k}u|^{2}d\mu_{k}\geq\frac{({\overline{N}-2})^{2}}{4}\int_{{\mathbb{R}}^{N}}
|\nabla_{k}u|^{2}d\mu_{k},
\end{equation}
where $\overline{N}:=N+2\gamma$, and the constant $\frac{({\overline{N}-2})^{2}}{4}$ is sharp.
\end{theorem}
\begin{proof}
Our goal is to find best constant $C$ satisfing\\
\begin{equation*}
\int_{{\mathbb{R}}^{N}}|x|^{2}|\Delta_{k}u|^{2}d\mu_{k}-C\int_{{\mathbb{R}}^{N}}|\nabla_{k}u|^{2}d\mu_{k}\geq0.
\end{equation*}
Using spherical decomposition:
\begin{equation*}
\int_{{\mathbb{R}}^{N}}|x|^{2}|\Delta_{k}u|^{2}d\mu_{k}=\sum_{n=0}^{\infty}\sum_{i=1}^{d(n)}\int_{0}^{+\infty}
\left(u^{''}_{n,i}+\frac{{\overline{N}}-1}{r}u^{'}_{n,i}+\frac{\lambda_{n}}{r^{2}}u_{n,i}\right)^{2}r^{{\overline{N}}+1}dr,
\end{equation*}
\begin{equation*}
\int_{{\mathbb{R}}^{N}}|\nabla_{k}u|^{2}d\mu_{k}=-\sum\limits_{n=0}^{\infty}\sum\limits_{i=1}^{d(n)}\int_{0}^{+\infty}\left(u^{''}_{n,i}+\frac{{\overline{N}}-1}{r}u^{'}_{n,i}+\frac{\lambda_{n}}{r^{2}}u_{n,i}\right)u_{n,i}r^{{\overline{N}}-1}dr.
\end{equation*}

By integration by parts, we have
\begin{multline*}
\int_{{\mathbb{R}}^{N}}|x|^{2}|\Delta_{k}u|^{2}d\mu_{k}-C\int_{{\mathbb{R}}^{N}}|\nabla_{k}u|^{2}d\mu_{k} \\
=\sum\limits_{n=0}^{\infty}\sum\limits_{i=1}^{d(n)}\int_{0}^{+\infty}\left(|u^{''}_{n,i}|^{2}r^{{\overline{N}}+1}-[({\overline{N}}-1)+2\lambda_{n}+C]|u^{'}_{n,i}|^{2}r^{{\overline{N}}-1} +(C\lambda_{n}+\lambda^{2}_{n})u^{2}_{n,i}r^{{\overline{N}}-3}\right)dr.
\end{multline*}
Let
\begin{equation*}
\begin{split}
I_{n,i}&=\int_{0}^{+\infty}|u^{''}_{n,i}|^{2}r^{{\overline{N}}+1}dr-[({\overline{N}}-1)+2\lambda_{n}+C]\int_{0}^{+\infty}|u^{'}_{n,i}|^{2}r^{{\overline{N}}-1}dr \\
&+(C\lambda_{n}+\lambda^{2}_{n})\int_{0}^{+\infty}u^{2}_{n,i}r^{{\overline{N}}-3}dr.
\end{split}
\end{equation*}

By using the following two weighted Hardy inequalities,
\begin{equation}
\int_{0}^{+\infty}|u^{'}|^{2}r^{{\overline{N}}+1}dr\geq\frac{{\overline{N}}^{2}}{4}\int_{0}^{+\infty}
u^{2}r^{{\overline{N}}-1}dr,
\end{equation}
\begin{equation}
\int_{0}^{+\infty}|u^{'}|^{2}r^{{\overline{N}}-1}dr\geq\frac{(\overline{N}-2)^{2}}{4}\int_{0}^{+\infty}
u^{2}r^{{\overline{N}}-3}dr.
\end{equation}

So
\begin{equation*}
\begin{split}
I_{n,i}&\geq\left(\frac{{\overline{N}}^{2}}{4}-(\overline{N}-1)-2\lambda_{n}-C\right)\int_{0}^{+\infty}|u^{'}_{n,i}|^{2}r^{{\overline{N}}-1}dr+(\lambda^{2}_{n}+C\lambda_{n})\int_{0}^{+\infty}u^{2}_{n,i}r^{{\overline{N}}-3}dr \\
&=\left(\frac{({\overline{N}}-2)^{2}}{4}-C-2\lambda_{n}\right)\int_{0}^{+\infty}|u^{'}_{n,i}|^{2}r^{{\overline{N}}-1}dr+\lambda_{n}(\lambda_{n}+C)\int_{0}^{+\infty}u^{2}_{n,i}r^{{\overline{N}}-3}dr.
   \end{split}
\end{equation*}

Let $C\leq\frac{({\overline{N}}-2)^{2}}{4}-2\lambda_{n}$, then we have
\begin{eqnarray*}
I_{n,i}&\geq&\left[\left(\frac{({\overline{N}}-2)^{2}}{4}-C-2\lambda_{n}\right){\frac{({\overline{N}}-2)^{2}}{4}}+\lambda_{n}(\lambda_{n}+C)\right]\int_{0}^{+\infty}u^{2}_{n,i}r^{{\overline{N}}-3}dr \\
&=&\left[\left(\frac{({\overline{N}}-2)^{2}}{4}-C\right){\frac{({\overline{N}}-2)^{2}}{4}}+\lambda_{n}\left(\lambda_{n}+C-\frac{({\overline{N}}-2)^{2}}{2}\right)\right]\int_{0}^{+\infty}u^{2}_{n,i}r^{{\overline{N}}-3}dr\geq0.
\end{eqnarray*}
Because $C\leq\frac{({\overline{N}}-2)^{2}}{4}-2\lambda_{n}$ and $C_{max}=\min\limits_{n}\{\frac{({\overline{N}}-2)^{2}}{4}-2\lambda_{n}\}=\frac{({\overline{N}}-2)^{2}}{4}$,
therefore
\begin{equation*}
\left(\frac{({\overline{N}}-2)^{2}}{4}-C\right){\frac{({\overline{N}}-2)^{2}}{4}}+\lambda_{n}\left(\lambda_{n}+C-\frac{({\overline{N}}-2)^{2}}{2}\right)\geq0.
\end{equation*}
Thus (\ref{eq4.2}) holds.
Then we show the optimality of $\frac{({\overline{N}-2})^{2}}{4}$. For any $\epsilon>o$, let
\begin{equation*}
u_{\epsilon}(r)=\left\{
\begin{array}{cc}
 1&,r<1,\\
r^{-\frac{\overline{N}-2+\epsilon}{2}} &,r>1.
\end{array}\right.
\end{equation*}
Straightforward calculation shows
\begin{equation*}
\lim\limits_{\epsilon\rightarrow0}\frac{\int_{{\mathbb{R}}^{N}}|x|^{2}|\Delta_{k}u|^{2}d\mu_{k}}{\int_{{\mathbb{R}}^{N}}|\nabla_{k}u|^{2}d\mu_{k}}=\frac{(\overline{N}-2)^{2}}{4}.
\end{equation*}
This completes the proof of Theorem \ref{T-4-1}.
\end{proof}

\begin{theorem}
Assume $N\geq{5}+2\gamma$. Then, for any $u\in{C_{0}^{\infty}} ({\mathbb{R}}^{N})$, we have the inequality
\begin{equation}\label{T-4-2}
\int_{{\mathbb{R}}^{N}}|\Delta_{k}u|^{2}d\mu_{k}\geq{\frac{{\overline{N}}^{2}}{4}}\int_{{\mathbb{R}}^{N}}
\frac{|\nabla_{k}u|^{2}}{|x|^{2}}d\mu_{k},
\end{equation}
where the constant $\frac{{\overline{N}}^{2}}{4}$ is sharp.
\end{theorem}
\begin{proof}
By integration by parts,
\begin{equation*}
\int_{{\mathbb{R}}^{N}}\frac{|\nabla_{k}u|^{2}}{|x|^{2}}d\mu_{k}=-\int_{{\mathbb{R}}^{N}}\frac{\Delta_{k}u\cdot{u}}{{|x|^{2}}}d\mu_{k}+2\int_{{\mathbb{R}}^{N}}u\frac{x\cdot{\nabla_{k}u}}{{|x|}^{4}}d\mu_{k},
\end{equation*}
where
\begin{equation}
\begin{split}
\int_{{\mathbb{R}}^{N}}u\frac{x\cdot{\nabla_{k}u}}{{|x|}^{4}}d\mu_{k}&=-\int_{{\mathbb{R}}^{N}}u\cdot{\nabla_{k}(\frac{xu}{|x|^{4}})} d\mu_{k} \\
&=-\int_{{\mathbb{R}}^{N}}u\left(\frac{N-4}{|x|^{4}}u+\frac{x}{|x|^{4}}\nabla_{k}u+\frac{2}{|x|^{4}}\sum\limits_{\alpha\in{R_{+}}}k_{\alpha}u(\sigma_{\alpha}x)\right) d\mu_{k}.
\end{split}
\end{equation}
Then
\begin{equation*}
\int_{{\mathbb{R}}^{N}}\frac{x\cdot{\nabla_{k}u}}{{|x|}^{4}}d\mu_{k}=-\frac{N-4}{2}\int_{{\mathbb{R}}^{N}}\frac{u^{2}}{|x|^{4}}d\mu_{k}-
\sum\limits_{\alpha\in{R_{+}}}k_{\alpha}\int_{{\mathbb{R}}^{N}}\frac{u(\sigma_{\alpha}x)u}{|x|^{4}}d\mu_{k}.
\end{equation*}
Therefore
\begin{equation*}
\begin{split}
\int_{{\mathbb{R}}^{N}}&\frac{|\nabla_{k}u|^{2}}{|x|^{2}}d\mu_{k}= -\int_{{\mathbb{R}}^{N}}\frac{\Delta_{k}u\cdot{u}}{|x|^{2}}d\mu_{k}
-(N-4)\int_{{\mathbb{R}}^{N}}\frac{u^{2}}{|x|^{4}}d\mu_{k}-2\sum\limits_{\alpha\in{R_{+}}}k_{\alpha}\int_{{\mathbb{R}}^{N}}
\frac{u(\sigma_{\alpha}x)}{|x|^{4}}d\mu_{k} \\
 &=-\int_{{\mathbb{R}}^{N}}\frac{\Delta_{k}u\cdot{u}}{|x|^{2}}d\mu_{k}-({\overline{N}}-4)\int_{{\mathbb{R}}^{N}}
\frac{u^{2}}{|x|^{4}}d\mu_{k}+2\sum\limits_{\alpha\in{R_{+}}}k_{\alpha}\int_{{\mathbb{R}}^{N}}
\frac{(u-u(\sigma_{\alpha}x))u}{|x|^{4}}d\mu_{k}.
\end{split}
\end{equation*}
Let
\begin{equation*}
u(x)=\sum\limits_{n=0}^{+\infty}\sum\limits_{i=1}^{d(n)}u_{n,i}(r)Y_{i}^{n}(\xi),
\end{equation*}
\begin{equation*}
u(\sigma_{\alpha}x)=\sum\limits_{n=0}^{+\infty}\sum\limits_{i=1}^{d(n)}\widetilde{u}_{n,i}(r)Y_{i}^{n}(\xi),
\end{equation*}
where
\begin{equation*}
u_{0,1}(r)=\frac{1}{\omega_{d}^{k}}\int_{\mathbb{S}^{N-1}}u(r\xi)\omega_{k}(\xi)d\nu(\xi),
\end{equation*}
Note that  $\omega_{k}(\xi)d\nu(\xi)$ is G invariant, by a change of variables $\sigma_{\alpha}\xi\rightarrow{\xi}$, we follow from (\ref{eq4.1}) that
\begin{equation*}
\widetilde{u}_{0,1}(r)=u_{0,1}(r).
\end{equation*}
Thus
\begin{equation*}
u-u(\sigma_{\alpha}x)=\sum\limits_{n=1}^{+\infty}\sum\limits_{i=1}^{d(n)}(u_{n,i}(r)-\widetilde{u}_{n,i}(r))Y_{i}^{n}(\xi).
\end{equation*}
From Parseval identity, we have
\begin{equation*}
\begin{split}
\int_{{\mathbb{R}}^{N}}\frac{1}{|x|^{4}}(u-u(\sigma_{\alpha}x))ud\mu_{k}&=\sum\limits_{n=1}^{+\infty}\sum\limits_{i=1}^{d(n)}\int_{0}^{+\infty}(u_{n,i}(r)-\widetilde{u}_{n,i}(r))
\cdot{u_{n,i}}r^{\overline{N}-5}dr \\
&=\int_{{\mathbb{R}}^{N}}\frac{1}{|x|^{4}}\left[(u-u_{0,1})-(u(\sigma_{\alpha}x)-\widetilde{u}_{0,1})\right](u-u_{0,1})d\mu_{k}.
\end{split}
\end{equation*}
Note that
\begin{equation*}
\begin{split}
-&\int_{{\mathbb{R}}^{N}}\frac{1}{|x|^{4}}(u(\sigma_{\alpha}x)-\widetilde{u}_{0,1})(u-u_{0,1})d\mu_{k}\\
\leq&\left(\int_{{\mathbb{R}}^{N}}\frac{1}{|x|^{4}}(u(\sigma_{\alpha}x)-\widetilde{u}_{0,1})^{2}d\mu_{k}\right)^{\frac{1}{2}}{{\left(\int_{{\mathbb{R}}^{N}}\frac{1}{|x|^{4}}(u-u_{0,1})^{2}d\mu_{k}\right)}^{\frac{1}{2}}} \\
=&\int_{{\mathbb{R}}^{N}}\frac{1}{|x|^{4}}(u-u_{0,1})^{2}d\mu_{k},
\end{split}
\end{equation*}
then we have
\begin{equation}
\int_{{\mathbb{R}}^{N}}\frac{1}{|x|^{4}}(u-u(\sigma_{\alpha}x))u d\mu_{k}\leq{2}\int_{{\mathbb{R}}^{N}}
\frac{1}{|x|^{4}}(u-u_{0,1})^{2}d\mu_{k}.
\end{equation}
By spherical h-harmonic decomposition,
\begin{multline*}
\int_{{\mathbb{R}}^{N}}\frac{{|\nabla_{k}u|}^{2}}{|x|^{2}}d\mu_{k}\leq-\int_{{\mathbb{R}}^{N}}\frac{u\cdot{\Delta_{k}u}}{|x|^{2}}d\mu_{k}-(\overline{N}-4)\int_{{\mathbb{R}}^{N}}
\frac{u^{2}}{|x|^{4}}d\mu_{k}+4\gamma\int_{{\mathbb{R}}^{N}}\frac{1}{|x|^{4}}(u-u_{0,1})^{2}d\mu_{k}\\
=-\sum\limits_{n=0}^{\infty}\sum\limits_{i=1}^{d(n)}\int_{0}^{+\infty}\left[u_{n,i}(u^{''}_{n,i}+\frac{\overline{N}-1}{r}u^{'}_{n,i}+\frac{\lambda_{n}}{r^{2}}u_{n,i})r^{\overline{N}-3}+(\overline{N}-4)\cdot{u^{2}_{n,i}}r^{\overline{N}-5}\right]dr \\
+4\gamma\sum\limits_{n=1}^{\infty}\sum\limits_{i=1}^{d(n)}\int_{0}^{+\infty}u^{2}_{n,i}r^{\overline{N}-5}dr\\
=\sum\limits_{n=0}^{\infty}\sum\limits_{i=1}^{d(n)}\int_{0}^{+\infty}\left[|u^{'}_{n,i}|^{2}r^{\overline{N}-3}-\lambda_{n}u^{2}_{n,i}r^{\overline{N}-5}\right]dr+4\gamma\sum\limits_{n=1}^{\infty}\sum\limits_{i=1}^{d(n)}\int_{0}^{+\infty}
u^{2}_{n,i}r^{\overline{N}-5}dr.
\end{multline*}
So
\begin{multline*}
\int_{{\mathbb{R}}^{N}}|\Delta_{k}u|^{2}d\mu_{k}-C\int_{{\mathbb{R}}^{N}}\frac{|\nabla_{k}u|^{2}}{|x|^{2}}d\mu_{k}\\
\geq\sum\limits_{n=0}^{\infty}\sum\limits_{i=1}^{d(n)}\int_{0}^{+\infty}\left[\left(u^{''}_{n,i}+\frac{\overline{N}-1}{r}u^{'}_{n,i}
+\frac{\lambda_{n}}{r^{2}}u_{n,i}\right)^{2}r^{\overline{N}-1}-C|u^{'}_{n,i}|^{2}r^{\overline{N}-3}+\lambda_{n}Cu^{2}_{n,i}r^{\overline{N}-5}\right]dr\\
-4C\gamma\sum\limits_{n=1}^{\infty}\sum\limits_{i=1}^{d(n)}\int_{0}^{+\infty}u^{2}_{n,i}r^{\overline{N}-5}dr\\
=\sum\limits_{n=0}^{\infty}\sum\limits_{i=1}^{d(n)}\int_{0}^{+\infty}\left[|u^{''}_{n,i}|^{2}r^{\overline{N}-1}+A_{n}|u^{'}_{n,i}|^{2}r^{\overline{N}-3}
+B_{n}u^{2}_{n,i}r^{\overline{N}-5}\right]dr.
\end{multline*}
By integration by parts, we obtain
\begin{equation*}
A_{n}=\overline{N}-2\lambda_{n}-1-C,
\end{equation*}
\begin{equation*}
B_{n}=\left\{\begin{array}{ll} \lambda_{0}(\lambda_{0}-2(\overline{N}-4)+C),& n=0;\\
    \lambda_{n}(\lambda_{n}-2(\overline{N}-4)+C)-4C\gamma,& n\geq1,\\
   \end{array}\right.
\end{equation*}
since $\lambda_{0}=0$, then $B_{0}=0$.\\
Using the following weighted Hardy inequality
\begin{equation}
\int_{0}^{+\infty}|u^{'}|^{2}r^{\overline{N}-1}dr\geq\frac{({\overline{N}-2})^{2}}{4}\int_{0}^{+\infty}u^{2}r^{\overline{N}-3}dr,
\end{equation}
\begin{equation}
\int_{0}^{+\infty}|u^{'}|^{2}r^{\overline{N}-3}dr\geq\frac{({\overline{N}-4})^{2}}{4}\int_{0}^{+\infty}u^{2}r^{\overline{N}-5}dr.
\end{equation}
Denote
\begin{equation*}
I_{n,i}=\int_{0}^{+\infty}\left[|u^{''}_{n,i}|^{2}r^{\overline{N}-1}+A_{n}|u^{'}_{n,i}|^{2}r^{\overline{N}-3}+
B_{n}u^{2}_{n,i}r^{\overline{N}-5}\right]dr,
\end{equation*}
then we have
\begin{equation}
I_{n,i}\geq\left[A_{n}+\frac{({\overline{N}-2})^{2}}{4}\right]\int_{0}^{+\infty}|u^{'}_{n,i}|^{2}r^{\overline{N}-3}dr+
B_{n}\int_{0}^{+\infty}u^{2}_{n,i}r^{\overline{N}-5}dr.
\end{equation}
For $n=0$,
\begin{equation*}
I_{0,1}\geq\left(\frac{{\overline{N}}^{2}}{4}-C\right)\int_{0}^{+\infty}|u^{'}_{0,1}|^{2}r^{\overline{N}-3}dr,
\end{equation*}
so we get $C\leq\frac{{\overline{N}}^{2}}{4}$.\\
For $n\geq1$, take $C=\frac{{\overline{N}}^{2}}{4}$, we get
\begin{equation*}
\begin{split}
I_{n,i}&\geq-2\lambda_{n}\int_{0}^{+\infty}|u^{'}_{n,i}|^{2}r^{\overline{N}-3}dr+B_{n}\int_{0}^{+\infty}u^{2}_{n,i}r^{\overline{N}-5}dr \\
&\geq\left[-2\lambda_{n}\frac{({\overline{N}-4})^{2}}{4}+\lambda_{n}\left(\lambda_{n}-2(\overline{N}-4)+\frac{{\overline{N}}^{2}}{4}\right)-{\overline{N}}^{2}\gamma\right]\int_{0}^{+\infty}u^{2}_{n,i}r^{\overline{N}-5}dr \\
&=D_{n}\int_{0}^{+\infty}u^{2}_{n,i}r^{\overline{N}-5}dr,
\end{split}
\end{equation*}
\noindent
here
\begin{equation*}
D_{n}:=\lambda_{n}\left(\lambda_{n}-\frac{{\overline{N}}^{2}-8\overline{N}}{4}\right)-{\overline{N}}^{2}\gamma.
\end{equation*}
So
\begin{equation*}
 D_{1}=\frac{(N-5-2\gamma){\overline{N}}^{2}+4}{4}.
\end{equation*}
When $N\geq5+2\gamma$, $D_{1}\geq0$.
\begin{equation*}
D_{2}=\frac{2N{\overline{N}}^{2}}{4}\geq0,
\end{equation*}
$D_{n}\geq{D_{2}}\geq0$, $(n=3,4,\ldots)$, so the inequality (\ref{T-4-2}) holds.

Next we prove the optimality of the constant $\frac{{\overline{N}}^{2}}{4}$. For $\forall\epsilon>0$, take
\begin{equation*}
u_{\epsilon}=\left\{
\begin{array}{ccc}
1  &,r\leq1, \\
r^{-\frac{N-4+\epsilon}{2}}&,r>1.
\end{array}
\right.
\end{equation*}
By directly computing
\begin{equation*}
\lim\limits_{\epsilon\rightarrow{0}}\frac{\int_{{\mathbb{R}}^{N}}|\Delta_{k}u_{\epsilon}|^{2}d\mu_{k}}{\int_{{\mathbb{R}}^{N}}\frac{|\nabla_{k}u_{\epsilon}|^{2}}{|x|^{2}}d\mu_{k}}
=\frac{{\overline{N}}^{2}}{4}.
\end{equation*}

\end{proof}

\textbf{Acknowledgements.} This work is  supported by the national natural science foundation of China (Grant No. 11771395 and 11571306).

\end{document}